\documentclass{mpaper}

\makeatletter
\newtheorem*{rep@theorem}{\rep@title}
\newcommand{\newreptheorem}[2]{%
\newenvironment{rep#1}[1]{%
\def\rep@title{#2 \ref{##1}}%
\begin{rep@theorem}}%
{\end{rep@theorem}}
}
\makeatother

\newreptheorem{theorem}{Theorem}

\include{thref}
\include{symbols}

\title{Minimal covers in the Weihrauch degrees}
\date{}

\author{Steffen Lempp}
\author{Joseph S.\ Miller}
\author{Arno Pauly}
\author{Mariya I.\ Soskova}
\author{Manlio Valenti}

\address[Steffen Lempp]{Department of Mathematics\\
University of Wisconsin - Madison\\
Madison, Wisconsin 53706\\
USA}
\email{\href{mailto:lempp@math.wisc.edu}{lempp@math.wisc.edu}}

\address[Joseph S.\ Miller]{Department of Mathematics\\
University of Wisconsin - Madison\\
Madison, Wisconsin 53706\\
USA}
\email{\href{mailto:jmiller@math.wisc.edu}{jmiller@math.wisc.edu}}

\address[Arno Pauly]{Department of Computer Science\\
Swansea University\\
Swan\-sea, SA1 8EN\\
UK}
\email{\href{mailto:arno.m.pauly@gmail.com}{arno.m.pauly@gmail.com}}

\address[Mariya I.\ Soskova]{Department of Mathematics\\
University of Wisconsin - Madison\\
Madison, Wisconsin 53706\\
USA}
\email{\href{mailto:msoskova@math.wisc.edu}{msoskova@math.wisc.edu}}

\address[Manlio Valenti]{Department of Mathematics\\
University of Wisconsin - Madison\\
Madison, Wisconsin 53706\\ 
USA}
\curraddr{Department of Computer Science\\
Swansea University\\
Swan\-sea, SA1 8EN\\
UK}
\email{\href{mailto:manliovalenti@gmail.com}{manliovalenti@gmail.com}}

\thanks{The first author was partially by AMS-Simons Foundation Collaboration Grant 626304. The second and fourth authors were partially supported by NSF Grant No.\ DMS-2053848. \\
\noindent\begin{minipage}{0.1\textwidth}\includegraphics[width=\textwidth]{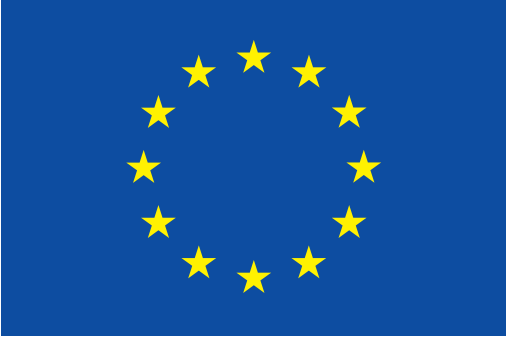}\end{minipage} \begin{minipage}{0.9\textwidth}This project has received funding from the European Unions Horizon 2020 research and innovation programme under the Marie Sklodowska-Curie grant agreement No 731143, \emph{Computing with Infinite Data}.\end{minipage}
}


\begin{document}

\maketitle

\begin{abstract}
In this paper, we study the existence of minimal covers and strong minimal covers in the Weihrauch degrees. We characterize when a
problem~$f$ is a minimal cover or strong minimal cover of a problem~$h$. We show that strong minimal covers only exist in the cone below~$\id$ and that
the Weihrauch lattice above~$\id$ is dense. From this, we conclude that the degree of $\id$ is first-order definable in the Weihrauch degrees and that the first-order theory of the Weihrauch degrees is computably isomorphic to third-order arithmetic.
\end{abstract}

\section{Introduction}

In a partial order $(\mathcal{D},\le)$, an element $\mathbf{a}$ is a
\textdef{minimal cover} of $\mathbf{b}$ if $\mathbf{b}<\mathbf{a}$ and there
is no $\mathbf{c}$ such that $\mathbf{b}<\mathbf{c}<\mathbf{a}$. In other
words, the interval between $\mathbf{a}$ and $\mathbf{b}$ is empty. We say
that $\mathbf{a}$ is a \textdef{strong minimal cover} of $\mathbf{b}$ if
$\mathbf{b}<\mathbf{a}$ and for all $\mathbf{c}$ if $\mathbf{c}<\mathbf{a}$
then $\mathbf{c}\leq \mathbf{b}$.

Understanding the properties and the distribution of minimal covers and
strong minimal covers can provide deep insights into the structure of a
partial order, and indeed there is an extensive literature on the
construction of minimal covers and strong minimal covers in the Turing
degrees. In contrast, despite its growing popularity, the structure of
Weihrauch degrees is vastly unexplored, as most of the efforts up to this
date have concentrated on the classification of the Weihrauch degrees of
specific problems. In particular, very little is known on the existence of
(strong) minimal covers in the Weihrauch degrees. In this paper, we fill this
gap by providing complete characterizations of minimal covers and strong
minimal covers in the Weihrauch degrees. We will see that this
analysis is then able to answer a number of other questions.

\subsection{Background}
Weihrauch reducibility \cite{Wei92,GM09,BGP17} classifies partial
multi-valued functions according to their uniform computational strength, and
it is often used to characterize the computability-theoretical complexity
of $\forall\exists$-statements. We briefly recall the main notions we need in this paper, and we refer the
reader to~\cite{BGP17} for a more thorough presentation on Weihrauch
reducibility. 

If $f$ and $g$ are partial multi-valued functions on Baire space (denoted
by $f,g\pmfunction{\Baire}{\Baire}$), we say that $f$ is \textdef{Weihrauch
reducible} to $g$, and write $f\weireducible g$, if there are two computable
functionals $\Phi\pfunction{\Baire}{\Baire}$ and $\Psi\pfunction{\Baire\times
\Baire}{\Baire}$ such that
\[ (\forall p\in \dom(f))[\; \Phi(p)\in \dom(g) \land (\forall q\in
g(\Phi(p)))\; \Psi(p,q)\in f(p) \;]. \] We use $(\weidegrees, \weireducible)$ to
denote the structure of degrees induced by Weihrauch reducibility. In the
computable analysis literature, Weihrauch reducibility is often defined in a
more general context, where $f$ and $g$ are partial multi-valued functions
between represented spaces. However, it is well-known that every Weihrauch
degree contains a representative with domain and codomain $\Baire$ (see, e.g., 
\cite[Lemma 11.3.8]{BGP17}). In other words, in order to study the structure
of the Weihrauch degrees, there is no loss of generality in restricting our
attention to computational problems on Baire space. In what follows, with a small abuse of notation we will identify a natural number $n$ with the infinite string constantly equal to $n$. 

The Weihrauch degrees are known to form a distributive lattice where join
and meet are induced, respectively, by the following operators:
\begin{itemize}
	\item $f\sqcup g$ is the problem with domain $\{0\}\times \dom(f)
	\cup \{1\}\times \dom(g)$ defined as
	\[ (f\sqcup g)(i,x):=\begin{cases} f(x) & \text{if }i = 0,\\ g(x) &
		\text{if }i=1.\end{cases}\]
	\item $f\sqcap g$ is the problem with domain $\dom(f) \times \dom(g)$
	defined as
	\[ (f\sqcap g)(x,z):= \{0\} \times f(x) \cup \{1\}\times g(z).\]
\end{itemize}
The degree of the empty function is a natural bottom element in the
Weihrauch degrees. The existence of a top element is equivalent to the
failure of a (relatively weak) form of choice. In particular, under
$\mathrm{ZFC}$, there is no top element in $(\weidegrees, \weireducible)$.

The statements and proofs of our main theorems exploit the interplay between
Weihrauch and Medvedev reducibility. For the sake of completeness, we recall
some basic facts on Medvedev reducibility and refer the reader to
\cite{Hinman2012,Sorbi1996} for more details. Given $A,B\subseteq \Baire$, we
say that $A$ is \textdef{Medvedev reducible} to $B$, and write
$A\medvedevreducible B$, if there is a computable functional
$\Phi\pfunction{\Baire}{\Baire}$ such that $\Phi(B)\subseteq A$. We write
$(\medvedevdegrees, \medvedevreducible)$ for the degree structure induced by
Medvedev reducibility. It is well-known that the Medvedev degrees form a
distributive lattice with a top element (the degree of $\emptyset$) and a
bottom element (the degree of $\Baire$).

There is a close connection between Weihrauch and Medvedev reducibility.
Indeed, we can rephrase the definition of Weihrauch reducibility as follows:
$f\weireducible g$ iff there are two computable functionals
$\Phi\pfunction{\Baire}{\Baire}$ and $\Psi\pfunction{\Baire\times
\Baire}{\Baire}$ such that $\Phi$ witnesses $\dom(g)\medvedevreducible
\dom(f)$ and, for every $p \in \dom(f)$, $\Psi(p,\cdot)$ witnesses $f(p)
\medvedevreducible g(\Phi(p))$. This suggests two possible embeddings of the
Medvedev degrees in the Weihrauch degrees~\cite{HiguchiPauly13}. For our
purposes, we explicitly mention the following one (see \cite[Section 5]{paulybrattka4} for a discussion of the other): For every $A\subseteq
\Baire$, let $d_A\function{A}{\{0^\mathbb{N}\}}$ be the constant function that
maps every element of $A$ to the constantly $0$ string. 
The map $d:=A\mapsto d_A$ induces a lattice embedding of
$\medvedevdegrees^{\mathrm{op}}=(\medvedevdegrees, \geqM)$ in $(\weidegrees,
\weireducible)$ \cite[Lemma 5.6]{HiguchiPauly13}.

A simple inspection reveals that the range of the embedding $d$ is exactly
the set of uniformly computable degrees. To see this, let
$\id\function{\Baire}{\Baire}$ be the identity function on Baire space.
It is immediate that a multi-valued function $f$ is uniformly
computable iff $f\weireducible \id$. In fact, writing $\id\restrict{X}$ for
the restriction of $\id$ to $X$, every problem $f\weireducible \id$ is
Weihrauch-equivalent to $\id\restrict{\dom(f)}$. In other words, the
uniformly computable problems are precisely those equivalent to one of the form
$\id\restrict{X}$ for some $X\subseteq \Baire$, which is, in turn, equivalent to $d_X$. So the lower cone of $\id$ is isomorphic to
$\medvedevdegrees^{\mathrm{op}}$, hence far from trivial. In particular, this implies that the
Medvedev degrees are first-order definable in $(\weidegrees, \weireducible,
\mathbf{1})$, where $\mathbf{1}$ is the Weihrauch degree of $\id$. The
question of whether $\mathbf{1}$ is first-order definable in $(\weidegrees,
\weireducible)$ was raised by Pauly during the conference ``Computability and
Complexity in Analysis 2020'' and the Oberwolfach meeting 2117 \cite{Oberwolfach2021}, see also~\cite{Pauly2020update}. Our results answer this question affirmatively.

The empty intervals in the Medvedev degrees have been fully characterized
in the literature. For every $p\in\Baire$, let $\medvedevsucc{p}:=\{
\str{e}\concat q \st \Phi_e(q)=p $ and $q\not\turingreducible p \}$.

\begin{theorem}[Dyment~{\cite[Cor.\ 2.5]{Dyment1976}}]
	For every $A\strictlymedvedevreducible B$, $B$ is a minimal cover of $A$ iff
	\[	(\exists p\in A)[\; A \medvedevequiv B \wedge \{p\} \text{ and } B
	\wedge \medvedevsucc{p} \medvedevequiv B\;], \]
	where $P \wedge Q:=\str{0}\concat P \cup \str{1}\concat Q$ is the meet in the
	Medvedev degrees.
\end{theorem}

The set $\medvedevsucc{p}$ is the immediate successor of $\{p\}$ in the
Medvedev degrees. In fact, the strong minimal covers  in
$\medvedevdegrees^{\mathrm{op}}$ are precisely those of the form $(\medvedevsucc{p}, \{p\})$. This implies that the property of being a degree
of solvability (i.e., being Medvedev equivalent to a singleton) is
first-order definable in $(\medvedevdegrees,\medvedevreducible)$ (Dyment~\cite[Cor.\
2.1]{Dyment1976}).

In particular, the fact that the lower cone of $\id$ (in the Weihrauch
degrees) is isomorphic to $\medvedevdegrees^{\mathrm{op}}$ immediately
yields:
\begin{corollary}
	For every $p\in\Baire$, $\id\restrict{\{p\}}$ is a strong minimal cover of
	$\id\restrict{\medvedevsucc{p}}$.	
\end{corollary}

Since $\NREC:=\{ q \st q\not\turingreducible 0\} \medvedevequiv
\medvedevsucc{0^\mathbb{N}}$, we also obtain:

\begin{corollary}\thlabel{thm:id_smc}
	$\id$ is a strong minimal cover of $\id\restrict{\NREC}$.	
\end{corollary}

\subsection{Our main theorems}

Unlike the Medvedev degrees, there are no results describing the structure of
minimal covers and strong minimal covers in the Weihrauch degrees. Recently,
Dzhafarov, Lerman, Patey, and Solomon~\cite{DamirLuminy} showed that no
Weihrauch degree can be minimal. This result can be obtained as a corollary
of our first main theorem:

\begin{theorem}\thlabel{thm:characterization_mc}
	Let $f$ and $h$ be partial multi-valued functions on Baire space. The following are equivalent:
	\begin{enumerate}
		\item
		$f$ is a minimal cover of $h$ in the Weihrauch degrees.
		\item
		$f \weiequiv h \sqcup \id\restrict{\{p\}}$ for some $p$ with $\dom(h)
		\not\medvedevreducible \{p\}$ and $ \dom(h)\medvedevreducible
		\medvedevsucc{p}$.
	\end{enumerate}
\end{theorem}

The second main theorem provides a similar characterization for strong
minimal covers:

\begin{theorem}\thlabel{thm:characterization_smc}
	Let $f$ and $h$ be partial multi-valued functions on Baire space. The following are equivalent:
	\begin{enumerate}
		\item
		$f$ is a strong minimal cover of $h$ in the Weihrauch degrees.
		\item
		There is $p\in\Baire$ such that $f \weiequiv \id\restrict{\{p\}}$ and $h
		\weiequiv \id\restrict{\medvedevsucc{p}}$.
	\end{enumerate}
\end{theorem}

A multi-valued function is called \textdef{pointed} if it has a computable
point in its domain. A computational problem $f$ is pointed iff
$\id\weireducible f$. In particular, the cone above $\id$ is exactly the
cone of pointed degrees. Using \thref{thm:characterization_mc}, we can
further characterize the non-pointed degrees.

\begin{corollary}\thlabel{cor:non-pointed}
	Let $g$ be a multi-valued function. The following are equivalent:
	\begin{enumerate}
		\item
		$\id \not\weireducible g$.
		\item
		There are $f,h$ such that  $g \weireducible h \strictlyweireducible f$ and
		$f$ is a minimal cover of $h$.
	\end{enumerate}
\end{corollary}

\begin{proof}
	The direction $(2)\Rightarrow (1)$ is straightforward  as 
	\thref{thm:characterization_mc} implies that the bottom of a minimal cover cannot have
	a computable point in its domain. To show that $(1)\Rightarrow (2)$, observe
	that if $\id \not\weireducible g$, then $g\weireducible g \sqcup \id
	\restrict{\NREC} \strictlyweireducible g \sqcup \id$. In particular, since $g
	\sqcup \id\restrict{\NREC} \sqcup \id\restrict{\{0^\mathbb{N}\}} \weiequiv g
	\sqcup \id$, by \thref{thm:characterization_mc}, the interval $g \sqcup
	\id\restrict{\NREC} \strictlyweireducible g \sqcup \id$ is empty.
\end{proof}

\begin{corollary}\thlabel{thm:dense_above_id}
	The pointed Weihrauch degrees are dense.
\end{corollary}

Our results provide two different first-order definitions of the degree of $\id$, thus the property of being uniformly computable is lattice-theoretic, answering the above-mentioned question by Pauly.

\begin{theorem}\thlabel{thm:definability_id}
The Weihrauch degree of $\id$ is first-order definable in $(\mathcal{W},\weireducible)$. In particular, it is both:
\begin{enumerate}
\item the greatest degree that is a strong minimal cover, and
\item the least degree such that the cone of Weihrauch degrees above it is dense.
\end{enumerate}
\end{theorem}
\begin{proof}
The first definition follows from \thref{thm:characterization_smc} and the fact that $\id$ is a strong minimal cover (\thref{thm:id_smc}). The second is immediate from \thref{cor:non-pointed}.
\end{proof}

Finally, the definability of $\id$ implies that the first-order theory of the Weihrauch degrees is computably isomorphic to the third-order theory of
arithmetic, and therefore it is ``as complicated as possible''.

\begin{theorem}
	The first-order theory of the Weihrauch degrees, the first-order theory of
	the Weihrauch degrees below $\id$, and the third-order theory of true
	arithmetic are pairwise recursively isomorphic.
\end{theorem}

\begin{proof}
	It is routine to check that Weihrauch reducibility between two multi-valued
	functions~$f$ and~$g$ can be defined using a $\Pi^1_2$ formula (with free
	third-order variables $f$ and $g$). This immediately implies that
	$\theory(\weidegrees(\le \id)) \le_1 \theory_3(\mathbb{N})$ and
	$\theory(\weidegrees) \le_1 \theory_3(\mathbb{N})$. The fact that the degree
	of $\id$ is first-order definable in $(\weidegrees, \weireducible)$
	(\thref{thm:definability_id}) yields $\theory(\weidegrees(\le \id)) \le_1
	\theory(\weidegrees)$. Since, as mentioned, the lower cone of $\id$ is
	isomorphic to $\medvedevdegrees^{\mathrm{op}}$, the statement follows from
	the fact that $\theory(\medvedevdegrees) \equiv_1 \theory_3(\mathbb{N})$
	(\cite[Thm.\ 3.13]{Shafer2011} and independently \cite[Thm.\ 2]{LNS09}).
\end{proof}

\section{Proof of the main theorems}

Before proving the main theorems, we need some preliminary results. The
following lemma is a step towards proving the first main theorem. It
implicitly shows that if $f$ is not a minimal cover of $h$, then there is a
uniform way to construct a problem $g$ such that $h\strictlyweireducible g
\strictlyweireducible f$.

\begin{lemma}\thlabel{thm:shuffling}
	Let $h \weireducible f$. If $f$ is a minimal cover of $h$ then there is $g$
	with $|\dom(g)|=1$ such that $f \weiequiv h \sqcup g$.
\end{lemma}

\begin{proof}
	We first outline the proof strategy: we construct a partial single-valued function
	$\xi\pfunction{\Baire}{\mathbb{N}}$ in stages, so that $\xi$ is defined on at
	most~$s$ points by stage~$s$. Let $F_\xi$ be the problem defined as
	$F_\xi(p,\xi(p)):=f(p)$. The function $\xi$ ``scrambles'' the domain of $f$: it
	is immediate that, for every choice of the function~$\xi$,
	$F_\xi\weireducible f$. The converse reduction trivially holds when $\xi$ is
	computable, but it does not hold in general. The construction attempts to
	build a function $\xi$ so that
	\begin{itemize}
		\item $F_\xi\not\weireducible h$
		\item $h \sqcup F_\xi\strictlyweireducible f$.
	\end{itemize}
	Since this would contradict our assumptions, we argue that the construction
	must fail. The failure of our construction will result in the desired
	function~$g$.
	
	The construction of $\xi$ proceeds as follows: for every stage $s$, we define
	a partial function $\xi_s$. We start the construction by letting
	$\xi_0:=\emptyset$. For the sake of readability, let us write
	$F_s:=F_{\xi_s}$.
	
	At stage $s+1=2\pairing{e,i}$, we extend $\xi_{s}$ so that $F_\xi
	\not\weireducible h$ via $\Phi_e,\Phi_i$. Since $\xi_{s}$ has finite domain
	and codomain $\mathbb{N}$, it has a total computable extension
	$\hat{\xi}\pfunction{\Baire}{\mathbb{N}}$. As observed, $F_{\hat{\xi}}
	\weiequiv f$, hence in particular $F_{\hat\xi} \not\weireducible h$ via
	$\Phi_e, \Phi_i$. By the definition of Weihrauch reducibility, there is some
	$p_0\in\dom(f)$ such that either $\Phi_e(p_0,\hat{\xi}(p_0)) \notin \dom(h)$
	or, for some $q\in h(\Phi_e(p_0,\hat{\xi}(p_0)))$,
	$\Phi_i((p_0,\hat{\xi}(p_0)),q) \notin
	F_{\hat\xi}(p_0,\hat{\xi}(p_0))=f(p_0)$. Fix some $p_0$ as above. Defining
	$\xi_{s+1}:=\xi_{s}\cup \{(p_0,\hat{\xi}(p_0))\}$ ensures that there is no
	extension $\xi'$ of $\xi_{s+1}$ such that $F_{\xi'} \weireducible h$ via the
	functionals $\Phi_e,\Phi_i$.
	
	At stage $s+1=2\pairing{e,i}+1$, we try to extend $\xi_s$ so that $f
	\not\weireducible h \sqcup F_\xi$ via $\Phi_e,\Phi_i$. The construction stops
	if $\Phi_e, \Phi_i$ witness $f\weireducible h \sqcup F_s$, as the same pair
	of functionals would witness the reduction $f\weireducible h \sqcup F_{\xi'}$ for any extension $\xi'$ of $\xi_s$. In this case, we simply define $\xi:=\xi_s$. Assume therefore that $\Phi_e,\Phi_i$ do not witness $f\weireducible h \sqcup F_s$, and
	let $p_1$ be such that either $\Phi_e(p_1)\notin\dom(h\sqcup F_s)$ or, for
	some $q \in (h\sqcup F_s)(\Phi_e(p_1))$, $\Phi_i(p_1,q)\notin f(p_1)$. If
	\begin{itemize}
		\item
		$\Phi_e(p_1)\uparrow$, or
		\item
		$\Phi_e(p_1)=(0,r)$ for some $r\notin\dom(h)$, or
		\item
		$\Phi_e(p_1)=(1,(p,k))$ for some $p$ such that $p\notin\dom(f)$ or
		$p\in\dom(\xi_s)$ with $\xi_s(p)\neq k$, or
		\item
		$\Phi_e(p_1)\in\dom(h\sqcup F_s)$ and for some $q \in (h\sqcup
		F_s)(\Phi_e(p_1))$, $\Phi_i(p_1,q)\notin f(p_1)$,
	\end{itemize}
	then there is no extension $\xi'$ of $\xi_s$ such that $f \weireducible h\sqcup F_{\xi'}$
	via the functionals $\Phi_e,\Phi_i$, hence we can just define
	$\xi_{s+1}:=\xi_s$. The remaining case is that $\Phi_e(p_1)=(1,(p,k))$ for
	some $p\in\dom(f)\setminus \dom(\xi_s)$. We define $\xi_{s+1}:=\xi \cup
	\{(p_1, k+1)\}$. Again, this ensures that there is no extension $\xi'$ of $\xi_{s+1}$
	such that $\Phi_e,\Phi_i$ witness that $f \weireducible h \sqcup
	F_{\xi'}$.
	
	Observe that, if for every $s+1=2\pairing{e,i}+1$, $f \not\weireducible h
	\sqcup F_s$, then, in the limit, we obtain a function $\xi$ such that
	$h\strictlyweireducible h\sqcup F_\xi \strictlyweireducible f$, against the
	assumption that $f$ is a minimal cover of $h$. This implies that, for some
	$s$ as above, $f \weireducible h \sqcup F_s = h\sqcup F_\xi$. Moreover,
	$\xi$, and hence $F_\xi$, has finite domain, which in turn implies that
	$f\weiequiv h\sqcup F_\xi$.
	
	Let $(p_i)_{i<n}$ be an enumeration of $\dom(\xi)$, and let $c_i$ be the function with domain $\{p_i\}$ defined as $c_i(p_i) := f(p_i)$. 
	Since $\dom(\xi)$ is finite, we immediately have $F_\xi
	\weireducible \bigsqcup_{i<n} c_i$. If, for all $i<n$, $c_i\weireducible h$,
	we would have $f \weireducible h \sqcup F_\xi\weireducible h$, which is a
	contradiction. Hence, for some $i<n$, $h\strictlyweireducible h\sqcup c_i
	\weireducible h\sqcup F_\xi \weiequiv f$. The fact that $f$ is a minimal
	cover of $h$ implies that $f \weiequiv h\sqcup c_i$. Letting $g:=c_i$
	concludes the proof.
\end{proof}

\begin{corollary}\thlabel{thm:smc_singleton}
	If $f$ is a strong minimal cover of $h$, then there is $g\weiequiv f$ such
	that $|\dom(g)|=1$.
\end{corollary}

\begin{proof}
	By \thref{thm:shuffling}, there is $g$ with $|\dom(g)|=1$ such that $f
	\weiequiv h \sqcup g$. Since $g\weireducible f$, the fact that $f$ is a
	strong minimal cover of $h$, implies that $g\weiequiv f$ or $g\weireducible
	h$. The latter is readily seen to yield a contradiction, as $g\weireducible
	h\Rightarrow h\strictlyweireducible f \weiequiv h \sqcup g \weiequiv h$,
	hence $g\weiequiv f$.
\end{proof}

For every set $A$, let $\chi_A$ denote the characteristic function of $A$.
\begin{lemma}
	\thlabel{thm:chi_D}
	If $f \not\weireducible h$, then there are at most countably many $D
	\subseteq \mathbb{N}$ such that $f \sqcap \chi_D \weireducible h$.
\end{lemma}

\begin{proof}
	We argue that each pair of potential reduction witnesses $\Phi,\Psi$ for $f
	\sqcap \chi_D \weireducible h$ can work for at most one choice of $D$. To see
	this, assume that $f \sqcap \chi_D \weireducible h$ via $\Phi,\Psi$. Notice
	that if there is $n_0\in\mathbb{N}$ such that
	\[
	(\forall x \in\dom(f))(\forall y \in h(\Phi(x,n_0)))(\exists z\in
	f(x))\; \Psi((x,n_0),y)=(0,z) ,
	\]
	then we would have $f\weireducible h$, against the hypothesis.
	
	This implies that, for every $D\subseteq \mathbb{N}$, if $\Phi,\Psi$ witness
	the reduction $f\sqcap \chi_D\weireducible h$ then for every $n\in\mathbb{N}$
	there exists some $x \in \dom(f)$ and some $y \in h(\Phi(x,n))$ such that
	$\Psi((x,n),y) = (1,\chi_D(n))$. In other words, membership of $n$ in $D$ is
	determined by the pair $\Phi,\Psi$, and hence the same pair cannot witness
	the reduction $f\sqcap \chi_E\weireducible h$ for any set $E\neq D$.
\end{proof}

In fact, if $f \not\weireducible h$ then there are exactly countably many
$D\subseteq \mathbb{N}$ such that $f \sqcap \chi_D \weireducible h$ if and only if
$\dom(h)\medvedevreducible \dom(f)$. (Otherwise, it is never the case that $f \sqcap \chi_D \weireducible h$.)

\begin{lemma}\thlabel{thm:non_empty_interval}
	If $f\not\weireducible \id$ and $f$ has singleton domain, then for all $h$
	such that $f\not\weireducible h$ there is $g \strictlyweireducible f$ such
	that $g\not\weireducible h$. It follows that $h\strictlyweireducible h\sqcup g
	\strictlyweireducible h \sqcup f$.
\end{lemma}

\begin{proof}
	Fix $f$ as above with $\dom(f) = \{x\}$ and let $h$ be such that $f\not
	\weireducible h$. By \thref{thm:chi_D}, there is $D\subseteq \mathbb{N}$ such that $f\sqcap \chi_D\not\weireducible h$. Let $g:=f\sqcap \chi_D$. Note that $g \strictlyweireducible f$ because every instance of $\chi_D$ (and hence of $g$) has computable solutions, while our assumptions that $f \not\weireducible \id$ and $|\dom(f)|=1$ ensure that $f$ does not have computable solutions. 
	
	To prove the last part of the statement, observe that the reductions
	$h\weireducible h\sqcup g \weireducible  h \sqcup f$ are immediate as
	$\sqcup$ is the join in the Weihrauch lattice, and $h\sqcup g \not
	\weireducible h$ follows immediately from $g\not\weireducible h$. A reduction
	$h\sqcup f \weireducible h\sqcup g$ would, in particular, yield
	$f\weireducible h \sqcup g$. This is a contradiction as $f$ is
	join-irreducible (since $|\dom(f)|=1$) and $f\not\weireducible h$ and
	$f\not\weireducible g$. This concludes the proof.	
\end{proof}

We are now able to prove the first main theorem, which we state again for the sake of readability.

\begin{reptheorem}{thm:characterization_mc}
	Let $f,h$ be partial multi-valued functions on Baire space. The following are equivalent:
	\begin{enumerate}
		\item
		$f$ is a minimal cover of $h$ in the Weihrauch degrees,
		\item
		$f \weiequiv h \sqcup \id\restrict{\{p\}}$ for some $p$ with $\dom(h)
		\not\medvedevreducible \{p\}$ and $ \dom(h)\medvedevreducible
		\medvedevsucc{p}$.
	\end{enumerate}
\end{reptheorem}

\begin{proof}
    $(1)\Rightarrow (2)$: By \thref{thm:shuffling}, there is $g$ with
	$\dom(g)=\{p\}$ such that $f\weiequiv h\sqcup g$. In particular
	$g\not\weireducible h$. If $g\not\weireducible \id \restrict{\{p\}}$ then
	$g$ satisfies the hypotheses of \thref{thm:non_empty_interval}, hence there
	is $G$ such that $h\strictlyweireducible h \sqcup G \strictlyweireducible h
	\sqcup g \weiequiv f$, contradicting the fact that $f$ is a minimal cover
	of $h$. Therefore $g\weiequiv \id \restrict{\{p\}}$. The fact that
	$\dom(h)\not\medvedevreducible \{p\}$ is immediate as
	\[
	\dom(h)\medvedevreducible \{p\}\Rightarrow
	\id\restrict{\{p\}}\weireducible h \Rightarrow h\weiequiv h\sqcup \id
	\restrict{\{p\}} \weiequiv f,
	\]
	contradicting $h\strictlyweireducible f$. Observe that $\dom(h)\not\medvedevreducible \medvedevsucc{p}$ would lead to a
	contradiction with the fact that $f$ is a minimal cover of $h$. Indeed, we
	would obtain
	\[
	h\strictlyweireducible h\sqcup \id\restrict{\medvedevsucc{p}}
	\strictlyweireducible h \sqcup \id\restrict{\{p\}}\weiequiv f,
	\]
	where the first two reductions are strict, respectively, since
	$\dom(h)\not\medvedevreducible \medvedevsucc{p}$ and
	$\dom(h)\cup\medvedevsucc{p}\not\medvedevreducible \{p\}$ (as
	$\medvedevsucc{p}\not\medvedevreducible \{p\}$).

    \medskip
	$(2)\Rightarrow (1)$: Assume towards a contradiction that there is $g$ such
	that $h\strictlyweireducible g \strictlyweireducible h \sqcup
	\id\restrict{\{p\}}$. The forward functional $\Phi$ of the reduction $g \weireducible h \sqcup	\id\restrict{\{p\}}$ lets us define two restrictions $g_0, g_1$ of $g$ by letting $\dom(g_i):=\{ p\in\dom(g) \st \Phi(g)(0)=i\}$. In particular, we obtain $g\weiequiv g_0\sqcup g_1$, $g_0\weireducible h$ (and hence $\dom(h)\medvedevreducible \dom(g_0)$), and $g_1\weireducible \id\restrict{\{p\}}$. The latter reduction implies $g_1 \weiequiv \id\restrict{A}$ for some $A\geqM \{p\}$, hence we obtain $h \strictlyweireducible g_0 \sqcup \id
	\restrict{A}\strictlyweireducible h \sqcup \id\restrict{\{p\}} $.

    Note that $h\strictlyweireducible g_0 \sqcup
	\id\restrict{A}$ implies that $\id\restrict{A}\not \weireducible h$, i.e.,
	$\dom(h)\not\medvedevreducible A$. This, in turn, implies that $A
	\medvedevequiv \{p\}$, as $\{p\}\strictlymedvedevreducible A$ would mean
	that $A\geqM \medvedevsucc{p} \geqM \dom(h)$, contradicting
	$\dom(h)\not\medvedevreducible A$. In other words, we obtain $h
	\strictlyweireducible g_0\sqcup \id\restrict{\{ p\}} \strictlyweireducible
	h \sqcup \id\restrict{\{ p\}}$. This is a contradiction, as $h \weireducible
	g_0 \sqcup \id\restrict{\{p\}}$ implies that $h \sqcup \id\restrict{\{p\}}
	\weireducible g_0 \sqcup \id\restrict{\{p\}}$.
\end{proof}

Observe that, since $\emptyset$ is the top of the Medvedev lattice and the
bottom of the Weihrauch lattice, the following is immediate.

\begin{corollary}[\cite{DamirLuminy}]
	There are no minimal degrees in $\weidegrees$. 
\end{corollary}

\thref{thm:characterization_mc} allows us to show also that there is a close
connection between antichains in the Turing degrees and minimal covers in the
Weihrauch degrees: for every family of pairwise Turing
incomparable sets $\{p_\alpha\}_{\alpha<\kappa}$ with $\kappa<2^{\aleph_0}$, there is a multi-valued function $h$ whose
minimal covers are exactly those of the form $h\sqcup \id\restrict{\{p_\alpha\}}$.

\begin{corollary}
	For every cardinal $\kappa \le 2^{\aleph_0}$, there is a problem $h$ with exactly $\kappa$
	minimal covers.
\end{corollary}

\begin{proof}
	The case for $\kappa=0$ follows from the fact that there are no minimal
	degrees in the Weihrauch lattice. It can also be proved using the fact that	the cone above $\id$ is dense (\thref{thm:dense_above_id}).
	
	Let $\{p_\alpha \in \Baire \st \alpha < \kappa\}$ be pairwise Turing
	incomparable and let $h$ be any problem with
	\[
	\dom(h):= \bigcup_{\alpha<\kappa} \{ q\in\Baire \st p_\alpha
	\strictlyturingreducible q\}.
	\]
	By \thref{thm:characterization_mc}, for every $\alpha<\kappa$, $h \sqcup \id\restrict{\{p_\alpha\}}$ is a minimal cover of $h$. Observe also that
	$h\sqcup \id \restrict{\{p_\alpha\}} \weireducible h \sqcup \id
	\restrict{\{p_\beta\}}$ implies  $p_\beta \turingreducible p_\alpha$ (as
	$\dom(h)\not\medvedevreducible \{p_\alpha\}$). This shows that $h$ has at
	least $\kappa$ minimal covers.
	
	\thref{thm:characterization_mc} implies that every $h$ has at most $2^{\aleph_0}$ minimal covers. To conclude the proof, assume that $k<2^{\aleph_0}$ and fix $p\in\Baire$ such that, for every $\alpha <	\kappa$, $p\not\turingequiv p_\alpha$. We show that $h\sqcup
	\id\restrict{\{p\}}$ is not a minimal cover of $h$. If there is
	$\alpha<\kappa$ such that $p_\alpha \turingreducible p$ then
	$\dom(h)\medvedevreducible \{p\}$ and hence $h\sqcup
	\id\restrict{\{p\}}\weireducible h$. If for every $\alpha<\kappa$, $p_\alpha \not \turingreducible p$, then there is some $q \strictlyturingabove p$ such that, for every $\alpha$, $p_\alpha \not \turingreducible q$. This follows from the fact that in the Turing degrees $p$ has $2^{\aleph_0}$  minimal covers  and for every $\alpha < \kappa$ there is at most one $m$ that is a minimal cover of $p$ such that $p_\alpha \turingreducible m$. Indeed, if $p_\alpha \turingreducible m$ and $m$ is a minimal cover of $p$ then $p\strictlyturingreducible p\oplus p_{\alpha}\turingreducible m$ implies $m\turingequiv p\oplus p_{\alpha}$. So $\kappa< 2^{\aleph_0}$ implies the existence of the desired $q$ as some minimal cover of $p$. 
It follows that $\dom(h)\not \medvedevreducible \medvedevsucc{p}$ and so  by \thref{thm:characterization_mc} we have that  $h\sqcup
	\id\restrict{\{p\}}$ is not a minimal cover of $h$.
\end{proof}

With a similar argument, we can also show the following:

\begin{corollary}
	For every $h$, if there is $p\in\Baire$ such that $\dom(h)\medvedevequiv
	\medvedevsucc{p}$, then $h$ has a unique minimal cover.
\end{corollary}

We finish with our characterization of strong minimal covers.

\begin{reptheorem}{thm:characterization_smc}
	Let $f,h$ be partial multi-valued functions on Baire space. The following are equivalent:
	\begin{enumerate}
		\item
		$f$ is a strong minimal cover of $h$ in the Weihrauch degrees,
		\item
		There is $p\in\Baire$ such that $f \weiequiv \id\restrict{\{p\}}$ and $h
		\weiequiv \id\restrict{\medvedevsucc{p}}$.
	\end{enumerate}
\end{reptheorem}

\begin{proof}
	To see that (2) implies (1), we just recall that $\{p\}$ is a strong minimal
	cover of $\medvedevsucc{p}$ in  $\medvedevdegrees^{\mathrm{op}}$, which is
	isomorphic to the lower cone of $\id$.
	
	We proceed to argue that (1) implies (2). As a strong minimal cover requires
	an empty interval, by \thref{thm:characterization_mc}, we can restrict
	ourselves to the case $f \weiequiv h \sqcup \id\restrict{\{p\}}$ where
	$\dom(h) \not\medvedevreducible \{p\}$ and $\dom(h) \medvedevreducible
	\medvedevsucc{p}$. As the top element of a strong minimal cover has to be
	join-irreducible, we find that $f\weireducible \id\restrict{\{p\}}$ (as
	$f\not\weireducible h$), and therefore $f\weiequiv \id\restrict{\{p\}}$.
	This, in turn, implies that $h \weiequiv \id\restrict{A}$ for some
	$A\subseteq \Baire$ such that $A\not\medvedevreducible \{ p\}$ and
	$A\medvedevreducible \medvedevsucc{p}$. Since $\{p\}\medvedevreducible A$ (as
	$h \weireducible f$), this yields $A \medvedevequiv \medvedevsucc{p}$ as
	claimed.
\end{proof}

In particular, this shows that every Weihrauch degree has a most one strong
minimal cover.

\bibliographystyle{mbibstyle}
\bibliography{bibliography}

\vspace*{-0.9cm}

\end{document}